\documentclass[12pt]{article}

\usepackage{color}

\usepackage{amssymb,amsthm,amsmath}
\usepackage{latexsym}
\usepackage{graphicx}
\usepackage{url}
\usepackage{fullpage}

\newtheorem{theorem}{Theorem}
\newtheorem{lemma}[theorem]{Lemma}
\newtheorem{claim}[theorem]{Claim}

\newtheorem{cor}[theorem]{Corollary}

\newtheorem{prop}[theorem]{Proposition}

\newcommand\ex{\ensuremath{\mathrm{ex}}}

\title{Counting copies of a fixed subgraph in $F$-free graphs}

\date{}

\author{
D\'aniel Gerbner\thanks{Hungarian Academy of Sciences, Alfr\'ed R\'enyi Institute of Mathematics, P.O.B. 127, Budapest H-1364, Hungary. e-mail: gerbner.daniel@renyi.mta.hu.}
\and
Cory Palmer\thanks{Department of Mathematical Sciences,
University of Montana, Missoula, Montana 59812, USA. e-mail: cory.palmer@umontana.edu}
}

\begin{document}

\maketitle

\let\thefootnote\relax\footnote{\textcopyright \hspace{.2em} 2019. This manuscript version is made available under the CC-BY-NC-ND 4.0 license \url{http://creativecommons.org/licenses/by-nc-nd/4.0/}}

\begin{abstract}
Fix graphs $F$ and $H$ and let $\ex(n,H,F)$ denote the maximum possible number of copies of the graph $H$ in an $n$-vertex $F$-free graph. The systematic study of this function was initiated by Alon and Shikhelman [{\it J. Comb. Theory, B}. {\bf 121} (2016)].
In this paper, we give new general bounds concerning this generalized Tur\'an function. We also determine $\ex(n,P_k,K_{2,t})$ (where $P_k$ is a path on $k$ vertices) and $\ex(n,C_k,K_{2,t})$ asymptotically for every $k$ and $t$. For example, it is shown that for $t \geq 2$ and $k\geq 5$ we have 	$\ex(n,C_k,K_{2,t})=\left(\frac{1}{2k}+o(1)\right)(t-1)^{k/2}n^{k/2}$.
We also characterize the graphs $F$ that cause the function $\ex(n,C_k,F)$ to be linear in $n$. In the final section we discuss a connection between the function $\ex(n,H,F)$ and Berge hypergraph problems.

\end{abstract}

	\noindent
	{\bf Keywords:} Tur\'an numbers, generalized Tur\'an numbers
	
	\noindent
	{\bf AMS Subj.\ Class.\ (2010)}: 05C35, 05C38

\section{Introduction}

Let $G$ and $F$ be graphs. We say that a graph $G$ is $F$-{\emph{free} if it contains no copy of $F$ as a subgraph. 
 Following Alon and Shikhelman \cite{alonsik}, let us denote the maximum number of copies of the graph $H$ in an $n$-vertex $F$-free graph by
\[
\ex(n,H,F).
\]
The case when $H$ is a single edge is the classical Tur\'an problem of extremal graph theory. 
In particular, the
 \emph{Tur\'an number} of a graph $F$ is the maximum number of edges possible in an $n$-vertex $F$-free graph $G$. This parameter is denoted $\ex(n,F)$ and thus $\ex(n,K_2,F) = \ex(n,F)$. For more on the ordinary Tur\'an number see, for example, the survey \cite{fusi}. Recall that the Tur\'an graph $T_{k-1}(n)$ is the complete $(k-1)$-partite graph with $n$ vertices such that the vertex classes are of size as close to each other as possible.

In 1949, Zykov \cite{zykov} proved that the Tur\'an graph $T_{k-1}(n)$ is the unique graph containing the maximum possible number of copies of $K_t$ in an $n$-vertex $K_{k}$-free graph (when $t<k$). By counting copies of $K_t$ in $T_{k-1}(n)$
we get the following corollary for $\ex(n,K_t,K_k)$. 
Let $\mathcal{N}(H,G)$ denote the number of copies of the subgraph $H$ in the graph $G$.

\begin{cor}[Zykov, \cite{zykov}]\label{erdos}
If $t < k$, then
\[\ex(n,K_t,K_k) = \mathcal{N}(K_t,T_{k-1}(n)) =\binom{k-1}{t}\left(\frac{n}{k-1}\right)^t + o(n^t).\]
\end{cor}

This result was rediscovered by Erd\H os \cite{Er1} and also follows from a theorem of Bollob\'as \cite{Bo} and the case $t=3$ and $k=4$ was known to Moon and Moser \cite{MoMo}.
Another proof appears in Alon and Shikhelman \cite{alonsik} modifying a proof of Tur\'an's theorem.

When $H$ is a pentagon, $C_5$, and $F$ is a triangle, $K_3$, the determination of $\ex(n,C_5,K_3)$ is a well-known conjecture of Erd\H os \cite{erpentvtr}. An upper bound of $1.03(\frac{n}{5})^5$ was proved by Gy\H ori \cite{Gy}. The blow-up of a $C_5$ gives a lower-bound of $(\frac{n}{5})^5$ when $n$ is divisible by $5$.
Hatami, Hladk\'y, Kr\'al', Norine and Razborov \cite{HaETAL} and independently Grzesik \cite{G2012} proved 
\begin{equation}\label{pentagons}
\ex(n,C_5,K_3) \leq \left(\frac{n}{5}\right)^5.
\end{equation}

Swapping the role of $C_5$ and $K_3$, we count the number of triangles in a pentagon-free graph. Bollob\'as and Gy\H ori \cite{BoGy} determined 

\[ (1+o(1))\frac{1}{3\sqrt{3}}n^{3/2} \leq \ex(n,K_3,C_5) \leq (1+o(1))\frac{5}{4}n^{3/2}.\]

The constant in the upper bound was improved to $\frac{\sqrt{3}}{3}$ by Alon and Shikhelman \cite{alonsik} and by Ergemlidze, Gy\H ori, Methuku and Salia \cite{C5C3}. 
 Gy\H ori and Li \cite{LiGy} give bounds on $\ex(n, K_3, C_{2k+1})$. A particularly interesting case to determine
the value of $\ex(n,K_3,K_{r,r,r})$ was posed by Erd\H os \cite{Er2} (second part of Problem 17) and remains open in general.

The systematic study of the function $\ex(n,H,F)$ was initiated by Alon and Shikhelman \cite{alonsik} who proved a number of different bounds. Two examples from their paper are as follows. An analogue of the K\H ov\'ari-S\'os-Tur\'an theorem
\[ \ex(n,K_3,K_{s,t}) = O(n^{3-3/s})\]
which is shown to be 
 sharp in the order of magnitude when $t > (s-1)!$ (see also \cite{KMV}).

Another example is an Erd\H os-Stone-Simonovits-type result that for fixed integers $t < k$ and a $k$-chromatic graph $F$ that
 \begin{equation}\label{A-S general}
 \ex(n, K_t, F) = \binom{k-1}{t}\left(\frac{n}{k-1}\right)^t + o(n^t).
 \end{equation}

Gishboliner and Shapira \cite{gs} determined the order of magnitude of $\ex(n,C_k,C_\ell)$ for every $\ell$ and $k\ge 3$. Moreover, they determined $\ex(n,C_k,C_4)$ asymptotically. We give a theorem (proved independently of the previous authors) that extends this result to $\ex(n,C_k,K_{2,t})$. Other results on the function $\ex(n,H,F)$ appear in \cite{gmv,ggmv,mq, GPS,letz,gk,cnr}. 

The goal of this paper is to determine new bounds $\ex(n,H,F)$ and investigate its behavior as a function. In Section~\ref{general section} we give general bounds using standard extremal graph theory techniques. In particular, we give the following extension of (\ref{A-S general}) using a modification of its proof in \cite{alonsik}.

\begin{theorem}\label{main}
Let $H$ be a graph and $F$ be a graph with chromatic number $k$, then
\[\ex(n,H,F) \leq \ex(n,H,K_k) + o(n^{|H|}).\]
\end{theorem}

Theorem~\ref{main} only gives a useful upper-bound if $\ex(n,H,K_k) = \Omega(n^{|H|})$, which happens if and only if $K_k$ is not a subgraph of $H$. 
For example, (\ref{A-S general}) follows by applying Corollary~\ref{erdos} in the case when $H=K_t$.
Applying (\ref{pentagons}) to the case when $H=C_5$ gives
\[
\ex(n, C_5, F) \leq \left(\frac{n}{5}\right)^5 + o(n^5)
\]
for every graph $F$ with chromatic number $3$. When $F$ contains a triangle, the construction giving the lower bound in (\ref{pentagons})
can be used to give an asymptotically equal lower bound on $\ex(n, C_5, F)$.

In Section~\ref{specific section} we determine $\ex(n,P_k,K_{2,t})$ and $\ex(n,C_k,K_{2,t})$ asymptotically.

It is natural to investigate which graphs $H$ and $F$ cause $\ex(n,H,F)$ to be linear in $n$.
When $H$ is $K_3$, Alon and Shikhelman \cite{alonsik} characterized the graphs $F$ with $\ex(n,K_3,F) = O(n)$.
In Section~\ref{linear section} we determine which graphs $F$ give $\ex(n,C_k,F)=O(n)$. Finally, in Section~\ref{berge section} we establish connections between this counting subgraph problem and Berge hypergraph problems.
For notation not defined in this paper, see Bollob\'as \cite{Bo-book}.

\section{General bounds on $\ex(n,H,F)$}\label{general section}

In this section we prove several general bounds on $\ex(n,H,F)$ using standard extremal graph theory techniques.
We begin with a proof of Theorem~\ref{main}. The proof mimics the proof of the Erd\H os-Stone-Simonovits by the regularity lemma. We use the following version of the regularity lemma and an embedding lemma found in \cite{Bo-book}. Recall that the {\it density} of a pair $A,B$ of vertex sets is $d(A,B)=e(A,B)/|A||B|$, where $e(A,B)$ is the number of edges between $A$ and $B$. Furthermore, the pair $A,B$ is {\it $\epsilon$-regular} if for any subsets $A'\subset A$, $B'\subset B$ with $|A'|\ge \epsilon |A|$ and $|B'|\ge \epsilon |B|$, we have $|d(A,B)-d(A',B')|< \epsilon$.

\begin{lemma}[Regularity Lemma]\label{regularity}
	For an integer $m$ and $0 < \epsilon < 1/2$ there exists an integer $M = M(\epsilon,m)$ such
	that every graph on $n \geq m$ vertices has a partition $V_0,V_1,\dots,V_r$ with $m \leq r \leq M$
	where $|V_0| < \epsilon n$, $|V_1| = |V_2| = \cdots = |V_r|$ and all but at most $\epsilon r^2$ of the pairs $V_i,V_j$, $1 \leq i < j \leq r$ are $\epsilon$-regular.
\end{lemma}

\begin{lemma}[Embedding lemma]\label{embedding}
	Let $F$ be a $k$-chromatic graph with $f\geq 2$ vertices.
	Fix $0<\delta < \frac{1}{k}$, let $G$ be a graph and let $V_1,\dots, V_k$ be disjoint sets of vertices of $G$.
	If each $V_i$ has $|V_i| \geq \delta^{-f}$ and each pair of partition classes is $\delta^f$-regular with density $\geq \delta + \delta^f$, then $G$ contains $F$ as a subgraph.
\end{lemma}

\begin{proof}[Proof of Theorem~\ref{main}]

	Fix $\delta > 0$ and an integer $m \geq k$ such that the following inequality holds
	\begin{equation}\label{alfa-bound}
	\left(\frac{1}{2m} + 2\delta^f + \frac{\delta + \delta^f}{2}\right) \mathcal{N}(H,K_{|H|}) < \alpha.
	\end{equation}

	Let us apply the regularity lemma with $\epsilon = \delta^f$ and $m$ to get $M = M(\epsilon,m)$. Let $G$ be a graph on  $n > M\delta^{-f}$ vertices and more than
	\[
	\ex(n,H,K_k) + \alpha n^{|H|}
	\]
	copies of $H$. We will show that $G$ contains $F$ as a subgraph.
	
	Let $V_0,V_1,\dots, V_r$ be the partition of $G$ given by the regularity lemma.
	We will remove the following edges.
	
	\begin{enumerate}
		\item Remove the edges inside of each $V_i$. There are at most $r \binom{n/r}{2} \leq \frac{n^2}{2r} \leq \frac{1}{2m}n^2$ such edges.
		
		\item Remove the edges between all pairs $V_i,V_j$ that are not $\epsilon$-regular. There are at most $\epsilon r^2$ such pairs and each has at most $(\frac{n}{r})^2$ edges. So we remove at most $\epsilon n^2$ such edges. 
		
		\item Remove the edges between all pairs $V_i,V_j$ if the density $d(V_i,V_j) < \delta+\delta^f$. 
		There are less than $\binom{r}{2} (\delta+\delta^f) (\frac{n}{r})^2 < \frac{\delta+\delta^f}{2} n^2$ such edges.
		
		\item Remove all edges incident to $V_0$. There are at most $\epsilon n^2$ such edges.
	\end{enumerate}
	
	In total we have removed at most 
	
	\[
	\left(\frac{1}{2m} + 2\delta^f + \frac{\delta + \delta^f}{2}\right)n^2
	\]
	edges.
	There are at most $\mathcal{N}(H,K_{|H|}) n^{|H|-2}$ copies of $H$
	containing a fixed edge.
	Therefore, by (\ref{alfa-bound}) we have removed less than $\alpha n^{|H|}$ copies of $H$.
	Thus, the resulting graph still has more than $\ex(n,H,K_k)$ copies of $H$ so it contains $K_k$ as a subgraph.
	
	The $k$ classes of the resulting graph that correspond to the vertices of $K_k$ satisfy the conditions of the embedding lemma so $G$ contains $F$.
\end{proof}

Using a standard first-moment argument of Erd\H os-R\'enyi \cite{Er-Re}  we can get a lower-bound on the number of copies of $H$ in an $F$-free graph.

\begin{prop} Let $F$ and $H$ be graphs such that $e(F) > e(H)$. Then
\[
\ex(n,H,F) = \Omega\left(n^{|H| - \frac{e(H)(|F|-2)}{e(F)-e(H)}}\right).
\]

\end{prop}

\begin{proof}
	Let $G$ be an $n$-vertex random graph with edge probability
	\[
	p = cn^{-\frac{|F|-2}{e(F)-e(H)}}
    \]	
	where $c = |H|^{|H|(e(F)-e(H))}+1$.
	
	Among $|F|$ vertices in $G$ there are at most $|F|!$ copies of the graph $F$. Therefore, the expected number of copies of $F$ is at most
	\[
	|F|!\binom{n}{|F|}p^{e(F)} \leq n^{|F|} p^{e(F)}.
	\]
	Fix $|H|$ vertices in $G$. The probability of a particular copy of $H$ appearing among those 			vertices is $p^{e(H)}$.
	Thus, the probability of at least one copy of $H$ appearing among those $|H|$ vertices is at 			least $p^{e(H)}$.
	Therefore, the expected number of copies of $H$ is at least
	\[	
	\binom{n}{|H|} p^{e(H)} \geq \left( \frac{n}{|H|} \right)^{|H|} p^{e(H)}.
	\]

	We remove an edge from each copy of $F$ in $G$ and count the remaining copies of $H$. 
	There are at most $n^{|H|-2}$
	copies of $H$ destroyed for each edge removed from $G$.
	
	Let $X$ be the random variable defined by the difference between the number of copies of $H$ and the number of copies of $H$ destroyed by the removal of edges.
	The expectation of $X$ is
	\[
	E[X] \geq \left( \frac{n}{|H|} \right)^{|H|} p^{e(H)}- n^{|H|-2}n^{|F|} p^{e(F)}.
	\]
	Which simplifies to
	\[
	E[X] = \Omega \left(n^{|H| - \frac{e(H)(|F|-2)}{e(F)-e(H)}}\right).
	\]
	This implies that there exists a graph such that after removing an edge from each copy of $F$ we are left with at least $E[X]$ copies of $H$.	
\end{proof}

We conclude this section with two simple bounds on $\ex(n,H,F)$. Neither result is likely to give a sharp bound, but may be useful as simple tools.

\begin{prop}\label{observ1} 
	$\ex(n,H,F) \ge \ex(n,F)-\ex(n,H)$.
\end{prop}

\begin{proof} 
	Consider an edge-maximal $n$-vertex $F$-free graph $G$. Remove an edge from each copy of the subgraph $H$ in $G$. The resulting graph does not contain $H$ and therefore has at most $\ex(n,H)$ edges. This means we have removed at least $\ex(n,F)-\ex(n,H)$ edges from $G$, thus $G$ (which is an $F$-free graph on $n$ vertices) contained at least $\ex(n,F)-\ex(n,H)$ copies of $H$.
\end{proof}

The other simple observation is a consequence of the Kruskal-Katona theorem \cite{Kru,Kat}. 
A hypergraph $\mathcal{H}$ is $k$-\emph{uniform} if all hyperedges have size $k$.
For a $k$-uniform hypergraph $\mathcal{H}$, the $i$-\emph{shadow} is the $i$-uniform hypergraph $\Delta_i \mathcal{H}$
whose hyperedges are the collection of all subsets of size $i$ of the hyperedges of $\mathcal{H}$.
We denote the collection hyperedges of a hypergraph $\mathcal{H}$ by $E(\mathcal{H})$.
Here we use a version of the Kruskal-Katona theorem due to Lov\'asz \cite{Lov}.

\begin{theorem}[Lov\'asz, \cite{Lov}]\label{KKL}
	If $\mathcal{H}$ is a $k$-uniform hypergraph and
	\[
	|E(\mathcal{H})| = \binom{x}{k} = \frac{x(x-1)\cdots (x-k+1)}{k!}
	\]
	for some real number $x \geq k$, then
	\[
	|E(\Delta_i \mathcal{H})| \geq \binom{x}{i}.
	\]
\end{theorem}

This gives the following easy corollary,

\begin{cor} 
	\[
	\ex(n,K_t,F) \leq \ex(n,F)^{t/2}.
	\]
\end{cor}

\begin{proof}
	Suppose $G$ is $F$-free and has the maximum number of copies of $K_t$. Let us consider the hypergraph $\mathcal{H}$ 
	whose hyperedges are the vertex sets of each copy of $K_t$ in $G$. Pick $x$ such that the number of hyperedges in $\mathcal{H}$ is
	\begin{equation}\label{kk-first}
	|E(\mathcal{H})| = \binom{x}{t}.
	\end{equation}
	Applying Theorem~\ref{KKL} we get that the $2$-uniform hypergraph (i.e., graph) $\Delta_2\mathcal{H}$ has size at least $\binom{x}{2}$.
	
	On the other hand, the family $\Delta_2\mathcal{H}$ is a subgraph of $G$. Therefore,
	\begin{equation}\label{kk-second}
	\binom{x}{2} \leq e(G) \leq \ex(n,F).
	\end{equation}
	Combining (\ref{kk-first}) and (\ref{kk-second}) gives the corollary.
\end{proof}

\section{Counting paths and cycles in $K_{2,t}$-free graphs}\label{specific section}

The maximum number of edges in a $K_{2,t}$-free graph is
\begin{equation}\label{k2t-bound}
\ex(n,K_{2,t}) = \left(\frac{1}{2} +o(1) \right) \sqrt{t-1}n^{3/2}.
\end{equation}

The upper bound above is given by K\H ov\'ari, S\'os and Tur\'an \cite{kst}
and the lower bound is given by an algebraic construction of F\"uredi \cite{Fu-graph}. We will refer to this construction as the \emph{F\"uredi graph} $F_{q,t}$.
We recall some well-known properties of $F_{q,t}$ without giving a full description of its construction. 
For fixed $t$ and $q$ a prime power such that $t-1$ divides $q-1$,
the graph $F_{q,t}$ has $n=(q^2-1)/(t-1)$ vertices. All but at most $2q$ vertices have degree $q$ and the others have degree $q-1$, thus the number of edges is $(1/2+o(1))\sqrt{t-1}n^{3/2}$. Furthermore, every pair of vertices has at most $t-1$ common neighbors while every pair of non-adjacent vertices has exactly $t-1$ common neighbors.

Alon and Shikhelman \cite{alonsik} used the F\"uredi graph to give a lower bound in the following theorem.
\begin{theorem}[Alon, Shikhelman, \cite{alonsik}]
	\[
\ex(n,K_3,K_{2,t}) = \left(\frac{1}{6} +o(1)\right) (t-1)^{3/2}n^{3/2}.
\]
\end{theorem}

We generalize this theorem to cycles of arbitrary length and paths. We use the notation $v_1v_2\cdots v_k$ for the path $P_k$ with vertices $v_1,\dots, v_k$ and edges $v_{i}v_{i+1}$ (for $1\leq i\leq k-1$). The cycle $C_k$ that includes this path and the edge $v_kv_1$ is denoted  $v_1v_2\cdots v_kv_1$.

\begin{prop}\label{c4s}
	For $t \geq 3$,
	\[
	\ex(n,C_4,K_{2,t}) = \left(\frac{1}{4} +o(1)\right) \binom{t-1}{2} n^2.
	\]
\end{prop}

\begin{proof}
	We begin with the upper bound. Consider an $n$-vertex graph $G$ that is $K_{2,t}$-free.
	Fix two vertices $u$ and $v$. As $G$ is $K_{2,t}$-free, $u$ and $v$ have at most $t-1$ common neighbors.
	Therefore the number of $C_4$s with $u$ and $v$ as non-adjacent vertices is at most $\binom{t-1}{2}$. Therefore, the  number of $C_4$s in $G$ is at most
	\[
	\frac{1}{2}\binom{n}{2} \binom{t-1}{2} \leq \frac{1}{4}\binom{t-1}{2} n^2
	\]
	as each cycle is counted twice.
	
	The lower bound is given by the F\"uredi graph $F_{q,t}$. Every pair of non-adjacent vertices has $t-1$ common neighbors, so there are $\binom{t-1}{2}$ copies of $C_4$ containing them. There are $(1/2+o(1))n^2$ pairs of non-adjacent vertices in $F_{q,t}$. Each $C_4$ is counted twice in this way, so the number of $C_4$s in $F_{q,t}$ is at least
	\[
	\
	\frac{1}{2}\left(\frac{1}{2}+o(1)\right) n^2 \binom{t-1}{2} 
	\]
	
\end{proof}

A slightly more sophisticated argument than the proof of Proposition~\ref{c4s} is needed to count longer cycles and paths.

\begin{theorem}\label{fredi} Fix $t\ge 2$. For $k \geq 5$,
	
	\[\ex(n,C_k,K_{2,t})=\left(\frac{1}{2k}+o(1)\right)(t-1)^{k/2}n^{k/2}\]
	 and for $k\geq 2$,
	\[\ex(n,P_k,K_{2,t})=\left(\frac{1}{2}+o(1)\right)(t-1)^{(k-1)/2}n^{(k+1)/2}.\]
	
\end{theorem}

\begin{proof} 
	We begin with the upper bound for $\ex(n,C_k,K_{2,t})$. 
	Let $G$ be a $K_{2,t}$-free graph. We distinguish two cases based on the parity of $k$.

	{\bf Case 1}:  $k$ is even.
	Fix a $(k/2)$-tuple $(x_1,x_2,\dots, x_{k/2})$ of distinct vertices of $G$. This can be done in at most $n^{k/2}$ ways. We count the number of cycles $v_1v_2\cdots v_kv_1$ such that $x_i = v_{2i}$ for $1 \leq i \leq k/2$.  As $G$ is $K_{2,t}$-free,
	there are at most $t-1$ choices for each vertex $v_{2i+1}$ on the cycle (for $0 \leq i \leq (k-2)/2$) as $v_{2i+1}$ must be joined to both $v_{2i+2}$ and $v_{2i}$ (where the indicies are modulo $k$).
	Each cycle $v_1v_2\cdots v_kv_1$ is counted by $2k$ different $(k/2)$-tuples, so the number of copies of $C_k$ is at most
	\[\left(\frac{1}{2k}\right)(t-1)^{k/2}n^{k/2}.\]

	{\bf Case 2}: $k$ is odd.
	Fix a $((k+1)/2)$-tuple $(x_1,x_2,\dots, x_{(k-3)/2}, y, z)$ of distinct vertices such that $yz$ is an edge.
	This can be done in at most \[
	2e(G)n^{(k-3)/2} \leq \left(1 +o(1)\right)(t-1)^{1/2}n^{3/2}n^{(k-3)/2} = \left(1 +o(1)\right)(t-1)^{1/2}n^{k/2}\]
	 ways by (\ref{k2t-bound}).
	We count the number of cycles $v_1v_2\cdots v_kv_1$ such that $x_i = v_{2i}$ for $1 \leq i \leq (k-3)/2$, $y = v_{k-1}$, and $z = v_{k}$. Similar to Case 1, as $G$ is $K_{2,t}$-free, there are at most $t-1$ choices for each of the $(k-1)/2$ remaining vertices $v_{2i+1}$ of the cycle.
	Each cycle $v_1v_2\cdots v_kv_1$ is counted by $2k$ different $((k+1)/2)$-tuples, so the number of copies of $C_k$ is at most
	\[\frac{1}{2k}(t-1)^{(k-1)/2} \left(1 +o(1)\right)(t-1)^{1/2}n^{k/2} = \left(\frac{1}{2k} +o(1)\right)(t-1)^{k/2}n^{k/2}.\]

	For the upper bound on $\ex(n,P_k,K_{2,t})$ we fix a tuple of distinct vertices of $G$ as above.
	We sketch the proof and leave the remaining details to the reader.
	If $k$ is odd we fix a $((k+1)/2)$-tuple $(x_1,x_2,\dots, x_{(k+1)/2})$ and if
	$k$ is even we fix a $((k+2)/2)$-tuple $(x_1,x_2,\dots, x_{(k-2)/2},y,z)$ such that $yz$ is an edge.
	In both cases we count the paths $v_1v_2 \cdots v_k$ such that $x_i = v_{2i-1}$ and with the additional conditions that $y = v_{k-1}$, and $z=v_k$ in the case $k$ even. Similar to the case for cycles there are at most $t-1$ choices for each of the remaining vertices of the path.
	Each path is counted exactly two times in this way.

	Both lower bounds are given by the F\"uredi graph $F_{q,t}$ for $q$ large enough compared to $t$ and $k$. We begin by counting copies of the path $P_k =v_1v_2\cdots v_k$ greedily. The vertex $v_1$ can be chosen in $n$ ways. As the F\" uredi graph $F_{q,t}$ has minimum degree $q-1$, we can pick vertex $v_i$ (for $i>1$) in at least $q-i+1$ ways. Each path is counted twice in this way,
	therefore, we have at least
	\[
	\frac{1}{2}n(q-k+1)^{k-1} = \left(\frac{1}{2} + o(1)\right)(t-1)^{(k-1)/2}n^{(k+1)/2}
	\]
	paths of length $k$ in the F\"uredi graph $F_{q,t}$.	
	
	For counting copies of the cycle $C_k = v_1v_2 \cdots v_kv_1$ we proceed as above with the addition that $v_k$ should be adjacent to $v_1$.
	In order to do this, we pick $v_1$ arbitrarily and $v_2,\dots, v_{k-3}$ greedily as in the case of paths. As $k \geq 5$ the vertex $v_{k-3}$ is distinct from $v_1$.
	From the neighbors of $v_{k-3}$ we pick $v_{k-2}$ that is not adjacent to $v_1$. The number of choices for $v_{k-2}$ is at least $q-k+3-(t-1)$ as $v_{k-3}$ and $v_1$ have at most $t-1$ common neighbors.
	From the neighbors of $v_{k-2}$ we pick $v_{k-1}$ that is not adjacent to any of the vertices $v_1,\dots, v_{k-3}$. Each $v_i$ has at most $t-1$ common neighbors with $v_{k-2}$ which forbids at most $(k-3)(t-1)$ vertices as a choice for $v_{k-1}$. Therefore, we have at least $q-k-2-(k-3)(t-1)$ choices for $v_{k-1}$. 
	
	Since $v_{k-1}$ is not joined to $v_1$ by an edge they have $t-1$ common neighbors and none of these neighbors are among $v_1,v_2,\dots, v_{k-1}$. Hence we can pick any of the common neighbors as $v_{k}$. Every copy of $C_k$ is counted $2k$ times, thus altogether we have at least 
	\[
	\frac{1}{2k}n(q-t(k-3))^{k-2}(t-1)=\left(\frac{1}{2k}+o(1)\right)(t-1)^{k/2}n^{k/2}
	\]
	 copies of $C_k$.
\end{proof}

\section{Linearity of the function $\ex(n,C_k,F)$}\label{linear section}

Recall that Alon and Shikhelman \cite{alonsik} characterized the graphs $F$ with $\ex(n,K_3,F) = O(n)$.
For trees they also essentially answer the question by determining the order of magnitude of $\ex(n,T,F)$ where both $T$ and $F$ are trees. One can easily see that their proof extends to the case when $F$ is a forest. On the other hand, if $F$ contains a cycle and $T$ is a tree, then $\ex(n,F)$ is superlinear and $\ex(n,T)$ is linear. Thus by Proposition \ref{observ1} we have that $\ex(n,T,F)$ is superlinear.

	\begin{figure}[h]
	\begin{center}
		\includegraphics[scale=1.3]{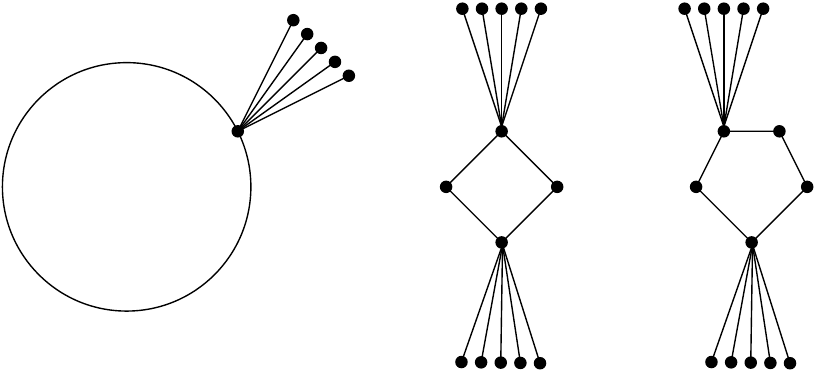}
	\end{center}
	\caption{The graphs $C_k^{*r}$, $C_4^{**r}$ and, $C_5^{**r}$}
	\label{c-star}
\end{figure}

Now we turn our attention to the case when $H$ is a cycle. We begin by introducing some notation. 
Let $C_k^{*r}$ be a cycle $C_k$ with $r$ additional vertices adjacent to a vertex $x$ of the $C_k$. For $k=4$, let  $C_4^{**r}$ be a cycle $v_1v_2v_3v_4$ with $2r$ additional vertices; $r$ are adjacent to $v_1$ and $r$ are adjacent to $v_3$. Similarly, let 
$C_5^{**r}$ be a cycle $v_1v_2v_3v_4v_5$ with $2r$ additional vertices; $r$ are adjacent to $v_1$ and $r$ are adjacent to $v_3$. See Figure~\ref{c-star} for examples of these graphs.


	\begin{figure}[h]
	\begin{center}
		\includegraphics[scale=1.3]{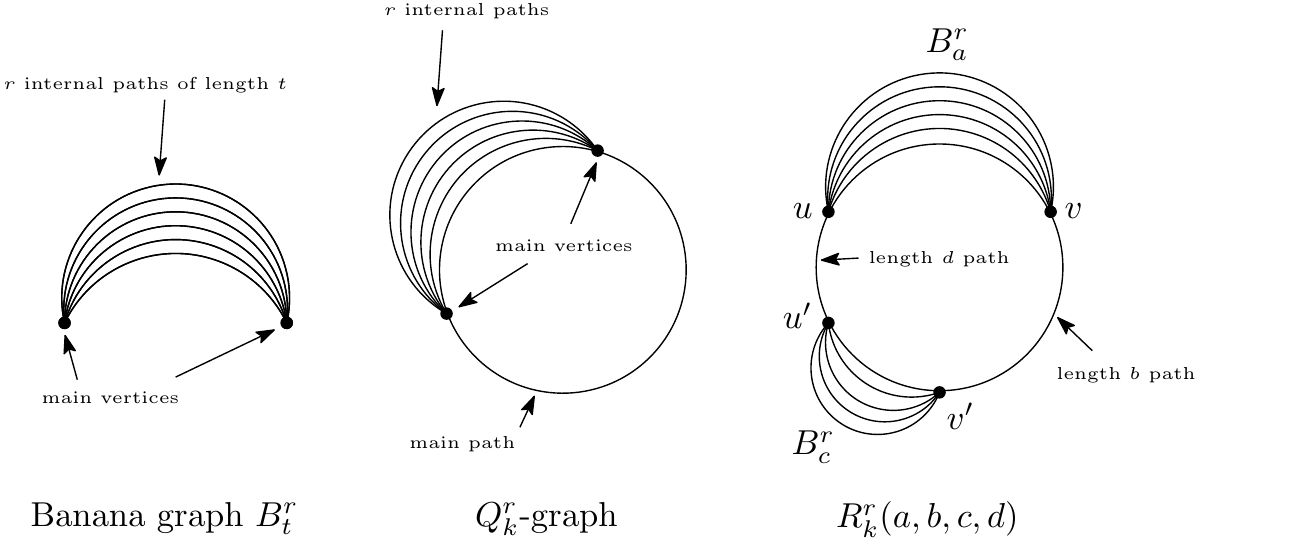}
	\end{center}
	\caption{A banana graph $B_t^r$, a $Q_k^r$-graph, and $R_k^r(a,b,c,d)$}
	\label{banana-et-al}
\end{figure}

A {\it banana graph} $B_t^r$ is the union of $r$ internally-disjoint $u$--$v$ paths of length $t$. We call the vertices $u,v$ the {\it main vertices} of $B_t^r$ and the $u$--$v$ paths of $B_t^r$ are its {\it internal paths}. 

For $t <k$, let $Q_k^r(t)$ be the graph consisting of a banana graph $B_t^r$ with main vertices $u,v$ and a $u$--$v$ path of length $k-t$ that is otherwise disjoint from $B_t^r$. Alternatively, $Q_k^r(t)$-graph is a $C_k$ with $r-1$ additional paths of length $t$ between two vertices that are joined by a path of length $t$ in the $C_k$. For simplicity, we call any graph $Q_k^r(t)$ a {\it $Q_k^r$-graph}. The {\it internal paths} and {\it main vertices} of a $Q_k^r$-graph are simply the internal paths and main vertices of the associated banana graph $B_t^r$. The {\it main path} of a $Q_k^r$-graph is the associated $u$--$v$ path of length $k-t$.
 
 For $a,c \geq 2$ and $b,d \geq 0$ such that $a+b+c+d=k$, let $R_k^r(a,b,c,d)$ be the graph formed by a copy of $B_a^r$ with main vertices $u,v$ and a copy of $B_c^r$ with main vertices $u',v'$ together with a $v$--$v'$ path of length $b$ and a $u$--$u'$ path of length $d=k-(a+b+c)$.
When $b=0$ we identify the vertices $v$ and $u'$ and when $d=0$ we identify the vertices $u$ and $v'$. Note that the last parameter $d$ is redundant, but we include it for ease of visualizing individual instances of this graph. 
For simplicity, we call any graph $R_k^r(a,b,c,d)$ an $R_k^r$-\emph{graph}.
Finally, call a graph $G$ an {\it $F_k^r$-graph} if $G$ is a forest and is a subgraph of every $R_k^r$-graph (for all permissible values of $a,b,c,d$).

We now characterize those graphs $F$ for which the function  $\ex(n,C_k,F)$ is linear.

\begin{theorem}\label{linear-main}
	For $k=4$ and $k=5$, if $F$ is a subgraph of $C_k^{**r}$ (for some $r$ large enough), then $\ex(n,C_k,F)=O(n)$.
	For $k>5$, if $F$ is a subgraph of $C_k^{*r}$ or an  ${F}_k^r$-graph (for some $r$ large enough), then $\ex(n,C_k,F)=O(n)$.
	On the other hand, for every $k>3$ and every other $F$ we have $\ex(n,C_k,F)=\Omega(n^2)$.
\end{theorem}

It is difficult to give a simple characterization of ${F}_k^r$-graphs. However, the following lemma gives some basic properties of these forests. For simplicity, the term {\it high degree} refers to a vertex of degree greater than $2$. 
A {\it star} is a single high degree vertex joined to vertices of degree $1$.
A {\it broom} is a path (possibly of a single vertex) with additional leaves attached to one of its end-vertices. 
Finally, let $c(F)$ be the sum of the number of vertices in the longest path in each component of $F$ (excluding the isolated vertex components).

	\begin{figure}[h]
	\begin{center}
		\includegraphics[scale=1.3]{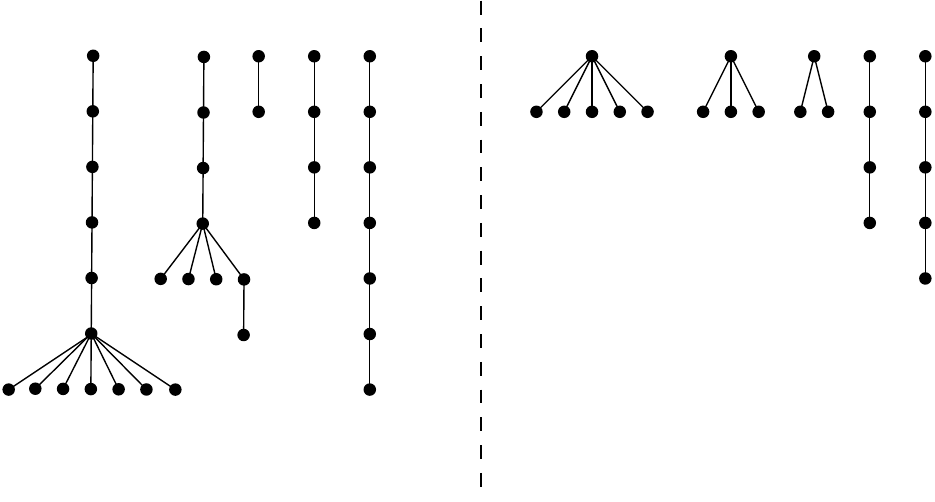}
	\end{center}
	\caption{An $F_k^r$-graph with a non-broom component and an $F_k^r$-graph $F$ with $c(F)=k+4$.}
\end{figure}

\begin{prop}\label{basic-props}
	Let $F$ be an ${F}_k^r$-graph, i.e., $F$ is a subforest of every $R_k^r$-graph. Then the following properties hold when $k>5$:
	\begin{enumerate}
		\item $F$ has at most two vertices of high degree. This implies that all but at most two components of $F$ are paths.
		
		\item Each component of $F$ has at most one vertex of high degree.
		
		\item Each vertex of high degree in $F$ is adjacent to at most two vertices of degree $2$.
		
		\item If $F$ has two high degree vertices, then at least one of them is contained in a component that is a broom.
		
		\item The number of vertices in the longest path in $F$ is at most $k$.
		
		\item $c(F) \leq k+4$.
		
		
		\item If $c(F) = k+4$, then $F$ contains three components that are stars on at least $3$ vertices. Furthermore, each component of $F$ with a high degree vertex is a star.
		
	\end{enumerate}
	
\end{prop}

\begin{proof}
	The first property follows as $F$ is a subgraph of the graph $R_k^r(2,0,k-2,0)$ which has exactly two high degree vertices.
	
	For property two, consider the graphs $R_k^r(2,0,k-2,0)$ and $R_k^r(3,0,k-3,0)$. Each graph has two high degree vertices and they are at distance $2$ and $3$, respectively. If $F$ had a component with two high degree vertices, then these vertices would be at distance $2$ and $3$ simultaneously; a contradiction. Note that we use $k>5$ here.
	
	For property three, consider the graph $R_k^r(2,0,2,k-4)$. This graph contains three high degree vertices $x,y,z$ such that every vertex adjacent to $y$ is adjacent to either $x$ or $z$. If $F$ has a component with a high degree vertex adjacent to more than two vertices of degree $2$, then that component contains a cycle; a contradiction. 
	
	For property four, again consider the graph $R_k^r(2,0,2,k-4)$ and denote the three high degree vertices $x,y,z$ as before. If $F$ has two components each with a high degree vertex, then without loss of generality one of these high degree vertices is $x$. If $x$ is adjacent to two vertices of degree $2$ in $F$, then one of these vertices is $y$. Therefore, the other high degree vertex in $F$ is $z$. The component containing $z$ cannot contain $y$, so $z$ is adjacent to at most one vertex of degree $2$, i.e., the component containing $z$ is a broom.
	
	For property five, observe that the number of vertices in a longest path in $R_k^r(2,0,2,k-4)$ is $k$.
	
	For property six and seven we can assume that all the components of $F$ are paths (by deleting unnecessary leaves) and that each component contains at least two vertices.
	
	Consider again the graph $R_k^r(2,0,2,k-4)$ with high degree vertices $x,y,z$ as above. Note that this graph contains an $x$--$z$ path on $k-3$ vertices. The components of $F$ containing $x$ or $z$ have at most $2$ additional vertices not on this path. Moreover, the component of $F$ containing $y$ has at most $3$ vertices not on this path (this includes $y$ itself). Therefore, $c(F) \leq k-3+2+2+3 = k+4$. This proves property six.
	In order to achieve equality $c(F)=k+4$ there must be three distinct components containing $x$, $y$ and $z$ and each of these components has $3$ vertices in their longest path, i.e., each such component is a star. This proves property seven.
\end{proof}

The next lemma establishes another class of graphs that contains each ${F}_k^r$-graph as a subgraph.

\begin{lemma}\label{two-As}
	Let $k>5$ and $H$ be a graph formed by two ${Q}_k^r$-graphs $Q_1,Q_2$ such that $Q_1$ and $Q_2$ share at most one vertex and such a vertex is on the main path of both $Q_1$ and $Q_2$. Then each ${F}_k^r$-graph is a subgraph of $H$.
\end{lemma}

\begin{proof}
	Let $F$ be an ${F}_k^r$-graph. We will show that $F$ can be embedded in $H$.
	Suppose $Q_1,Q_2$ share a vertex $x$ on their main paths as this is the more difficult case.

Let $F'$ be a graph formed by components of $F$ such that $c(F') \leq k$ and there is at most one vertex of high degree in $F'$. We claim that $F'$ can be embedded into $Q_2$.  Indeed, first we embed the component of $F'$ containing the high degree vertex using a main vertex of $Q_2$. The remaining (path) components of $F'$ can be embedded into the remaining vertices of the $C_k$ in $Q_2$ greedily. Now, if we can embed components of $F$ into $Q_1$ without using the vertex $x$ such that the remaining components satisfy the conditions of $F'$ above, then we are done.

First suppose that $c(F) = k+4$. By property seven of Proposition~\ref{basic-props} let $T$ and $T'$ be distinct star components of $F$ such that $T$ has exactly $3$ vertices. It is easy to see that $T$ and $T'$ can both be embedded to $Q_1$ without using the vertex $x$. 
Therefore, the remaining components of $F$ can be embedded into $Q_2$.

We may now assume $c(F) \leq k+3$. If $F$ contains a single component, then it can be embedded into $Q_1$ by property five of Proposition~\ref{basic-props}. If $F$ has no high degree vertex, then every component is a path. In this case it is easy to embed $F$ into $Q_1$ and $Q_2$. So let us assume that $F$ contains at least two components and at least one high degree vertex.

The graph $Q_1$ has two high degree vertices. Therefore, one of them is connected to $x$ by a path $P_\ell$ with $\ell>(k+2)/2$. 

Suppose $F$ contains two high degree vertices, then let $T$ be a component containing a high degree vertex. We may assume that the number of vertices on the longest path in $T$ is at most $(k+3)/2$ (as there are two components with a high degree vertex). Therefore, we may embed $T$ into $Q_1$ without using the vertex $x$. The remaining components of $F$ can be embedded into $Q_2$.

Now suppose $F$ contains exactly one high degree vertex. If $F$ contains a component with longest path on $k$ vertices, then
it can be embedded into $Q_2$ and the remaining component of $F$ can be embedded into $Q_1$ without using vertex $x$.
 So we may assume all components in $F$ have longest paths with less than $k$ vertices. If there is a (path) component on at least three vertices, then it can be embedded into $Q_1$ without using the vertex $x$ and the remaining components of $F$ can be embedded into $Q_2$. If there is no such path, then all path components are single edges. Two such edges can be embedded into $Q_1$ without using the vertex $x$ and the remaining components can be embedded into $Q_2$ as before.
\end{proof}

 A version of the next lemma has already appeared in a slightly different form in \cite{gkpp}.

	\begin{figure}[h]
	\begin{center}
		\includegraphics[scale=1.3]{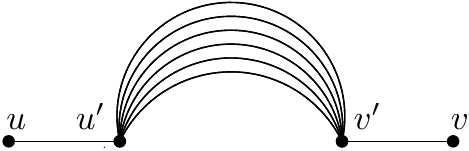}
	\end{center}
	\caption{The graph $B$ from Lemma~\ref{ugly-lemma}}
	\label{undir-figure}
\end{figure}

\begin{lemma}\label{ugly-lemma} 
	Fix integers $s \geq 2$ and $i \geq 2$. Let $G$ be a graph containing a family $\mathcal{P}$ of $(si)^{2i-2}$ $u$--$v$ paths of length $i$.
	Then $G$ contains a subgraph $B$ consisting of
	a banana graph $B_t^s$ (for some $t \leq i$) with main vertices $u',v'$ together with a $u$--$u'$ path and and $v'$--$v$ path that are disjoint from each other and otherwise disjoint from $B$ (we allow that the additional paths be of length $0$, i.e., $u=u'$ and $v=v'$),
	such that  each $u$--$v$ path in $B$ is a sub-path of some member of $\mathcal{P}$.
	Moreover, if each member of $\mathcal{P}$ is a sub-path of some copy of $C_k$ in $G$, then $G$ contains a $Q^{s'}_k$-graph where $s' = s-k$.
\end{lemma}

\begin{proof}
	We prove the first part of the lemma by induction on $i$. The statement clearly holds for $i=2$, as such a collection of paths is a banana graph. Let $i>2$ and suppose the lemma holds for smaller values of $i$.
	If there are $s$ disjoint paths of length $i$ between $u$ and $v$ then we are done. So we may assume that there are at most $s-1$ disjoint paths of length $i$ from $u$ to $v$. The union of a set of disjoint paths of length $i$ from $u$ to $v$ has at most $si$ vertices. Furthermore, every other $u$--$v$ path of length $i$ must intersect this set of vertices. Therefore, there is a vertex $w$ that is contained in at least $(si)^{2i-3}$ of these paths. This $w$ can be in different positions in those paths, but there are at least $(si)^{2i-4}$ paths where $w$ is the $(p+1)$st vertex (counting from $u$) with $1\le p<i$. Then there are either at least $(si)^{2p-2}> (sp)^{2p-2}$ sub-paths of length $p$ from $u$ to $w$ or at least $(si)^{2(i-p)-2}> (s(i-p))^{2(i-p)-2}$ sub-paths of length $i-p$ from $w$ to $v$.
	Without loss of generality, suppose there are at least $(si)^{2p-2}> (sp)^{2p-2}$ sub-paths of length $p$ from $u$ to $w$. Then, by induction on this collection of paths of length $p<i$, we find a banana graph $B_t^s$ with main vertices $u',v'$ together with a $u'$--$u$ path and a $v'$--$w$ path (that are disjoint from each other).
	As there is a path from $w$ to $v$ we have the desired subgraph $B$. 	
	
	Now it remains to show that if each member of $\mathcal{P}$ is a sub-path of some copy of $C_k$ in $G$, then $G$ contains a $Q^{s'}_k$-graph. Suppose we have a graph $B$ from the first part of the lemma. Let $C$ be a cycle of length $k$ that contains any $u$--$v$ path of length $i$ in $B$. Note that $C$ also contains a $u$--$v$ path $P$ of length $k-i$.
	The internal vertices of $P$ intersect at most $k$ of the internal paths of the banana graph $B_t^s$ in $B$. Remove these internal paths from $B_t^s$ and let $B'$ be the resulting subgraph of $B$. Now $B'$ together with $P$ forms a ${Q}_k^{s'}$-graph.
\end{proof}

\begin{proof}[Proof of Theorem~\ref{linear-main}] 
	
	First let us suppose that $F$ is a graph such that $\ex(n,C_k,F) = o(n^2)$. 
	Therefore, $F$ must be a subgraph of every graph with $\Omega(n^2)$ copies of $C_k$.  
	It is easy to see that each $R_k^r$-graph contains $\Omega(n^2)$ copies of $C_k$. Thus, $F$ is a subgraph of every ${R}_k^r$-graph.
	
	The F\"uredi graph $F_{q,2}$ does not contain a copy of $C_4$ and contains $\Omega(n^2)$ copies of $C_k$ for $k \geq 5$. Furthermore, $F_{q,3}$ contains $\Omega(n^2)$ copies of $C_4$. This follows from the proof of the lower bound in Theorem~\ref{fredi}. Therefore, when $k\geq 5$ and $r$ and $q$ are large enough, $F$ is a subgraph of $F_{q,2}$.
	When $k=4$, and $r$ and $q$ are large enough, $F$ is a subgraph of $F_{q,3}$.
	
	\begin{claim}\label{one-cycle}
		The graph $F$ contains at most one cycle and it is of length $k$.	
	\end{claim}
	
	\begin{proof}
		As $F$ is a subgraph of $R_k^r(2,0,2,k-4)$, every cycle in $F$ is of length $k$ or $4$. If $k>4$, as $F$ is a subgraph of $F_{q,2}$, it does not contain cycles of length $4$. Therefore, all cycles in $F$ are of length $k$.
		
		Suppose there is more than one copy of $C_k$ in $F$. For $k=4$, as $F$ is a subgraph of $R_4^r(2,0,2,0)$ it is easy to see that any two copies of $C_4$ in $F$ form a $K_{2,3}$ or $K_{2,4}$. This contradicts the fact that $F$ is also a subgraph of $F_{q,3}$.	For $k=5$, as $F$ is a subgraph of $R_5^r(2,0,3,0)$ it is easy to see that any two copies of $C_5$ in $F$ form a $C_4$ or $C_6$. This contradicts the fact that all cycles are of length $k$.
		For $k>5$, as $F$ is a subgraph of 
		 $R_k^r(2,0,2,k-4)$, every pair of $C_k$s in $F$ share $k-3$ or $k-1$ vertices.  
		 On the other hand, as $F$ is a subgraph of $R_k^r(3,0,3,k-6)$, every pair of 
		 $C_k$s in $F$ share $k-4$ or $k-2$ vertices; a contradiction.	
	\end{proof}

We distinguish three cases based on the value of $k$.

{\bf Case 1}: $k=4$. The graph $F$ is a subgraph of $R_{4}^r(2,0,2,0)$. 
By Claim~\ref{one-cycle}, $F$ has at most one cycle. The subgraphs of $R_{4}^r(2,0,2,0)$ with at most one cycle are clearly subgraphs of $C_4^{**2r}$.

{\bf Case 2}:  $k=5$.  The graph $F$ is a subgraph of $R_5^r(2,0,3,0)$ and therefore has at most $2$ vertices of degree greater than $2$ and they are non-adjacent. Furthermore, $F$ is a subgraph of $R_5^r(2,0,2,1)$. By Claim~\ref{one-cycle}, $F$ has at most one cycle. The subgraphs of $R_5^r(2,0,2,1)$ with at most one cycle that are simultaneously subgraphs of $R_5^r(2,0,3,0)$ are subgraphs of $C_5^{**2r}$.

{\bf Case 3}: $k>5$. 
First assume that $F$ is a forest. As every ${R}_k^r$-graph contains $\Omega(n^2)$ copies of $C_k$, each must contain $F$ as a subgraph. Therefore, $F$ is an ${F}_k^r$-graph by definition.

Now consider the remaining case when $F$ contains a cycle $C$.
As $F$ is a subgraph of $R_k^r(2,0,2,k-2)$, every edge of $F$ is incident to $C$.
If $F$ has at least two vertices of degree greater than $2$ on $C$, then as $F$ is a subgraph of both $R_k^r(2,0,k-2,0)$ and $R_k^r(3,0,k-3,0)$, we have that these two vertices should be at distance $2$ and $3$ from each other in $F$; a contradiction. Thus, there is only one vertex of degree greater than $2$ on $C$.
 Therefore, $F$ is a subgraph of $C_k^{*r}$.

This completes the first part of the proof that if $F$ is a graph such that $\ex(n,C_k,F)=O(n)$, then $F$ is as characterized in the theorem.

\bigskip

Now it remains to show that if $F$ is as characterized in the theorem, then $\ex(n,C_k,F)<cn$ for some constant $c$. 
The constants $k$ and $r$ are given by the statement of the theorem.
Fix constants $r'',r',\gamma,c',c$ in the given order such that each is large enough compared to $k$, $r$ and the previously fixed constants.

Let $G$ be a vertex-minimal counterexample, i.e, $G$ is an $n$-vertex graph with at least $cn$ copies of $C_k$ and no copy of $F$ such that $n$ is minimal. We may assume every vertex in $G$ is contained in at least $c$ copies of $C_k$, otherwise we can delete such a vertex (destroying fewer than $c$ copies of $C_k$) to obtain a smaller counterexample. 

{\bf Case 1}: $F$ contains a cycle.
Thus, for $k=4,5$ we have that $F$ is a subgraph of $C_k^{**r}$ and for $k>5$ we have that $F$ is a subgraph of $C_k^{*r}$.
If every vertex of $G$ has degree at least $2r+k$, then on any $C_k$ in $G$ we can build a copy of $C_k^{*r}$ or $C_k^{**r}$ greedily. These graphs contain $F$; a contradiction.

Now let $x$ be a vertex of degree less than $2r+k$. This implies that there is an edge $xy$ contained in at least $c/(2r+k)$ copies of $C_k$. Therefore, there are at least $c/(2r+k)$ $x$--$y$ paths of length $k-1$. 
As $c$ is large enough compared to $k$ and $r$, we may apply Lemma~\ref{ugly-lemma} to this collection of paths of length $k-1$ (each a subgraph of a $C_k$) to get a $Q^r_k$-graph $Q$. The graph $Q$ contains $C^{*r}_k$ when $k>5$ and  $C^{**r}_k$ when $k=4,5$. This implies that $G$ contains $F$; a contradiction.

{\bf Case 2}: $F$ is a forest. Note that if $k=4,5$, then $F$ is a subgraph of $C_k^{**r}$, and we are done by the same argument as in Case 1. Thus, we may  assume $k>5$.

\begin{claim}\label{repeat-claim}
	Suppose $G$ contains a collection  $\mathcal{C}$ of at least $c'n$ copies of $C_k$. 
	Then there is an integer $\ell <k$ such that $G$ contains a $Q_k^{r'}$-graph $Q$ with main vertices $x,y$ and internal paths of length $\ell$ such that less than $c'n$ members of $\mathcal{C}$ contain $x,y$ at distance $\ell$.
\end{claim}

\begin{proof}
	We distinguish two cases.
	
	{\bf Case 1:} There exists two vertices $u,v$ of $G$ in at least $c'n$ members of $\mathcal{C}$.
	 Then there are at least $(c'/k)n$ members of $\mathcal{C}$ 
	 that contain a $u$--$v$ sub-path of length $i <k$. Let us suppose that $u$ and $v$ are chosen such that $i$ is minimal.
	Among these $u$--$v$ paths of length $i$ we can find a collection $\mathcal{P}$ of $(c'/k)n/(ni)\ge c'/k^2$ of them that contain some fixed vertex $w$ (different from $u$ and $v$) such that $w$ is at distance $j<i$ from $u$ in each such $u$--$v$ path.
	Applying Lemma~\ref{ugly-lemma} this collection $\mathcal{P}$ of $u$--$w$ paths of length $j$ gives a $Q_k^{r'}$-graph $Q$. Let $x,y$ be the main vertices of $Q$ and let $\ell \leq j<i$ be the length of the main paths in $Q$.
 By the minimality of $i$, there are less than $c'n$ members of $\mathcal{C}$ that contain $x,y$ at distance $\ell$.

	{\bf Case 2:} The graph $G$ does not contain two vertices in $c'n$ members of $\mathcal{C}$.
	As $G$ is $F$-free and $F$ is a forest there are at most $2|V(F)|n$  edges in $G$. Thus, there is an edge $uv$ contained in at least $c'/(2|V(F)|)$ members of $\mathcal{C}$. Let $\mathcal{P}$ be a collection of $c'/(2|V(F)|)$ $u$--$v$ paths of length $k-1$ defined by these members of $\mathcal{C}$.
	Applying Lemma~\ref{ugly-lemma} to $\mathcal{P}$ gives a $Q_k^{r'}$-graph $Q$. Let $x,y$ be the main vertices of $Q$.
	By the assumption in Case 2, the vertices $x,y$ are contained in less than $c'n$ total copies of $C_k$ in $G$. 
\end{proof}

	Now let us apply Claim~\ref{repeat-claim} repeatedly in the following way. Let $\mathcal{C}_0$ be the collection of all $cn$ copies of $C_k$ in $G$. We may apply Claim~\ref{repeat-claim} to $\mathcal{C}_0$ to find a $Q_k^{r'}$-graph $Q_1$ with main vertices $x_1,y_1$ at distance $\ell_1$ in $Q_1$. Now remove from $\mathcal{C}_0$ the copies of $C_k$ that contain $x,y$ at distance $\ell_1$ and let $\mathcal{C}_1$ be the remaining copies of $C_k$ in $\mathcal{C}_0$. Note that $|\mathcal{C}_1| \geq (c-c')n$ and that none of the copies of $C_k$ in $Q_1$ are present in $\mathcal{C}_1$.
Repeating the argument above on $\mathcal{C}_1$ in place of $\mathcal{C}_0$ gives another $Q_k^{r'}$-graph $Q_2$ with main vertices $x_2,y_2$. We can continue this argument until we have $k\gamma$ different $Q_k^{r'}$-graphs (as $c$ is large enough compared to $c'$).

A pair of vertices $x,y$ can appear as main vertices in at most $k$ of the graphs $Q_1,Q_2,\dots, Q_{k\gamma}$. Indeed, as once they appear as main vertices at distance $\ell \leq k$ in some $Q_k^{r'}$-graph we remove all copies of $C_k$ that have $x,y$ at distance $\ell$.
Therefore, there is a collection of $\gamma=k\gamma/k$ different $Q_k^{r'}$-graphs such that no two of the $Q_k^{r'}$-graphs have the same two main vertices. Let $Q_1',Q_2',\dots, Q_\gamma'$ be this collection of $Q_k^{r'}$-graphs.

The internal paths of any $Q_i'$ may share vertices with $Q_j'$ (for $j \neq i$). However, for $r'$ large enough compared to $r''$, we may remove internal paths from each of the $Q_i'$s to construct a collection of $Q_k^{r''}$-graphs $Q_1'',Q_2'',\dots, Q_\gamma''$ such that each pair $Q_i''$, $Q_j''$ only share vertices on their respective main paths (for $i \neq j$).

Now let $M_1,M_2,\dots, M_\gamma$ be the collection of main paths of the $Q_k^{r''}$-graphs $Q_1'',Q_2'',\dots, Q_\gamma''$.
If there are two paths $M_i$ and $M_j$ that share at most one vertex, then we may apply Lemma~\ref{two-As} to $Q_i''$ and $Q_j''$ to find a copy of $F$ in $G$; a contradiction.

So we may assume that each $M_i$ shares at least two vertices with each other $M_j$. Recall that  $Q_i''$ and $Q_j''$ share at most one of their main vertices. Therefore, there is a vertex $u \in M_1$ that is contained in at least $\gamma/k$ of the paths $M_2,M_3,\dots, M_\gamma$. Moreover, $u$ is the $i$th vertex in at least $\gamma/k^2$ of those paths. Each of these paths contains another vertex from $M_1$. At least $\gamma/k^3$ of them contain the same vertex $v$, and it is the $j$th vertex in at least $\gamma/k^4$ of them. Thus, there are at least $\gamma/k^4$ $u$--$v$ paths of length $|j-i|$. As $\gamma/k^4$ is large enough we may apply Lemma~\ref{ugly-lemma} to this collection of $u$--$v$ paths of length $|j-i|$ to get a subgraph $B$ consisting of a banana graph $B_t^{r}$ (for some $t <k$) with main vertices $u',v'$ together with a $u$--$u'$ path and a $v$--$v'$ path. Each $u$--$v$ path of $B$ is a sub-path of some $Q_i''$. Pick any such $Q_i''$ and take its union with $B$.
The vertices of $B$ intersect at most $kr$ internal paths of $Q_i''$. As $r''$ is large enough compared to $r$,  we may remove internal paths of $Q_i''$ that intersect the vertices of $B$ to get a graph containing an $R_k^{r}$-graph. As $F$ is a subgraph of every $R_k^r$-graph, we have that $G$ contains $F$; a contradiction.
\end{proof}

\section{Connection to Berge-hypergraphs}\label{berge section}
The problem of counting copies of a graph $H$ in an $n$-vertex $F$-free graph is closely related to the study of Berge hypergraphs. Generalizing the notion of hypergraph cycles due to Berge, the authors introduced \cite{gp1} the notion of Berge copies
of any graph. Let $F$ be a graph.  We say that a hypergraph $\mathcal{H}$ is a
{\it Berge}-$F$ if there is a 
bijection $f : E(F) \rightarrow E(\mathcal{H})$ such that 
$e \subseteq f(e)$ for every $e \in E(F)$. Note that Berge-$F$ actually denotes a class of hypergraphs. The maximum number of hyperedges in an $n$-vertex hypergraph
with no sub-hypergraph isomorphic to any Berge-$F$ is denoted $\ex(n,\textrm{Berge-}F)$.
When we restrict ourselves to $r$-uniform hypergraphs, this maximum is denoted $\ex_r(n, \textrm{Berge-}F)$.

Results of Gy\H{o}ri, Katona and Lemons \cite{GyKaLe} together with Davoodi, Gy\H{o}ri, Methuku and Tompkins \cite{DaETAL} 
give tight bounds on $\ex_r(n, \textup{Berge-}P_\ell)$.
Upper-bounds on 
$\ex_r(n,\textup{Berge-}C_{\ell})$ are given by Gy\H{o}ri and Lemons \cite{GyLe4} when $r\geq 3$. A brief survey of Tur\'an-type results for Berge-hypergraphs can be found in Subsection 5.2.2 in \cite{gp}.

An early link between counting subgraphs and Berge-hypergraph problems was established by Bollob\'as and Gy\H ori \cite{BoGy} who investigated both $\ex_3(n,\textrm{Berge-}C_5)$ and $\ex(n,K_3,C_5)$. The connection between these two parameters is also examined in two recent manuscripts \cite{danietal, coryetal}. In this section we prove two new relationships between these problems.

\begin{prop}\label{Bergecontainment} 
	Let $F$ be a graph. Then
	  \[\ex(n,K_r,F)\le \ex_r(n,\textup{Berge-}F) \le \ex(n,K_r,F)+\ex(n,F).\]
	  and
	\[\ex(n,\textup{Berge-}F)=\max_{G}  \left\{\sum_{i=0}^n \mathcal{N}(K_i,G) \right\} \leq \sum_{i=0}^n \ex(n,K_i,F)\]
where the maximum is over all $n$-vertex $F$-free graphs $G$.   
\end{prop}

\begin{proof} Given an $F$-free graph $G$, let us construct a hypergraph $\mathcal{H}$ on the vertex set of $G$ by replacing each clique of $G$ by a hyperedge containing exactly the vertices of that clique. The hypergraph $\mathcal{H}$ contains no copy of a Berge-$F$. This gives 
	$\ex(n,K_r,F)\le \ex_r(n,\textup{Berge-}F)$ and
	
	\[
	\max_{G}  \left\{\sum_{i=0}^n \mathcal{N}(K_i,G) \right\} \leq \ex(n,\textup{Berge-}F)
	\]
	where the maximum is over all $n$-vertex $F$-free graphs $G$.  
	
Given an $n$-vertex hypergraph $\mathcal{H}$ with no Berge-$F$ subhypergraph, we construct a graph $G$ on the vertex set of $\mathcal{H}$ as follows. Consider an order $h_1, \dots, h_k$ of the hyperedges of $\mathcal{H}$ such that the hyperedges of size two appear first.
We proceed through the hyperedges in order and at each step try to choose a pair of vertices in $h_i$ to be an edge in $G$. If no such pair is available, then each pair of vertices in $h_i$ is already adjacent in $G$. In this case, we add no edge to $G$. A copy of $F$ in $G$ would correspond exactly to a Berge-$F$ in $\mathcal{H}$, so $G$ is $F$-free.

 For each hyperedge $h_i$ where we did not add an edge to $G$, there is a clique on the vertices of $h_i$ in $G$. Thus, the number of hyperedges of $\mathcal{H}$ is at most the number of cliques in $G$. 
 If $\mathcal{H}$ is $r$-uniform, then each hyperedge $h_i$ of $\mathcal{H}$ corresponds to either an edge in $G$ or a clique $K_r$ on the vertices of $h_i$ (when we could not add an edge to $G$). Therefore, the number of hyperedges in $\mathcal{H}$ is at most
$\ex(n,K_r,F)+\ex(n,F)$.
\end{proof}

As in the case of traditional Tur\'an numbers we may forbid multiple hypergraphs.
In particular, 
let $\ex_r(n,\{\textup{Berge-}F_1,\textup{Berge-}F_2,\dots, \textup{Berge-}F_k\})$ denote the maximum number of hyperedges in an $r$-uniform $n$-vertex hypergraph with no subhypergraph isomorphic to any $\textup{Berge-}F_i$ for all $1 \leq i \leq k$. Similarly, $\ex(n,H,\{F_1,F_2,\dots, F_k\})$ denotes the maximum number of copies of the graph $H$ in an $n$-vertex graph that contains no subgraph $F_i$ for all $1 \leq i \leq k$.

\begin{prop}\label{lvas} For $k\ge 4$,
	\[\ex_3(n, \{\textup{Berge-}C_2, \dots, \textup{Berge-}C_k\})=\ex(n, K_3, \{ C_4,\dots, C_k\}).\]

\end{prop}

\begin{proof} 
	Let $\mathcal{H}$ be an $n$-vertex $3$-uniform hypergraph with no Berge-$C_i$ for $i=2,3,\dots,k$ and the maximum number of hyperedges.
	Consider the graph $G$ on the vertex set of $\mathcal{H}$ where a pair of vertices are adjacent if and only if they are contained in a hyperedge of $\mathcal{H}$. As $\mathcal{H}$ is $C_2$-free (i.e., each pair of hyperedges share at most one vertex) each edge of $G$ is contained in exactly one hyperedge of $\mathcal{H}$.
	
	Each hyperedge of $\mathcal{H}$ contributes a triangle to $G$. We claim that $G$ contains no other cycles of
	length $i$ for $i=3,4,5,\dots, k$. That is, $G$ contains no cycle with two edges coming from different hyperedges of $\mathcal{H}$.
	Suppose (to the contrary) that $G$ does contain such a cycle $C$.
	If two edges of $C$ come from the same hyperedge, then they are incident in $C$. Therefore, these two edges can be replaced by the edge between their disjoint endpoints (which is contained in the same hyperedge) to get a shorter cycle. We may repeat this process until we are left with a cycle such that each edge comes from a different hyperedge of $\mathcal{H}$. Then this cycle corresponds exactly to a Berge-cycle of at most $k$ hyperedges in $\mathcal{H}$; a contradiction. Thus, $\ex_3(n, \{\textup{Berge-}C_2, \dots, \textup{Berge-}C_k\})\leq\ex(n, K_3, \{ C_4,\dots, C_k\})$.
	
	On the other hand, let $G$ be an $n$-vertex graph with no cycle $C_4,C_5,\dots, C_k$ and the maximum number of triangles. Construct a hypergraph $\mathcal{H}$ on the vertex set of $G$ where the hyperedges of $\mathcal{H}$ are the triangles of $G$. The graph $G$ is $C_4$-free, so each pair of triangles share at most one vertex, i.e., $\mathcal{H}$ contains no $\textup{Berge-}C_2$. If $\mathcal{H}$ contains a $\textup{Berge-}C_3$, then it is easy to see that $G$ contains a $C_4$; a contradiction.
	
	Therefore, if $\mathcal{H}$ contains any $\textup{Berge-}C_i$ for $i=4,\dots, k$, then $G$ contains a cycle $C_i$; a contradiction.
	Thus, $\ex_3(n, \{\textup{Berge-}C_2, \dots, \textup{Berge-}C_k\})\geq\ex(n, K_3, \{ C_4,\dots, C_k\})$.
\end{proof}

Alon and Shikhelman \cite{alonsik} showed that for every $k > 3$,
$
\ex(n, K_3,\{C_4,\dots, C_k\})\ge \Omega(n^{1+\frac{1}{k-1}}).
$
For $k=4$ they showed that $\ex_3(n,K_3, C_4)=(1+o(1))\frac{1}{6}n^{3/2}$. Lazebnik and Verstra\"ete \cite{LaVe} proved $\ex_3(n,\{\textup{Berge-}C_2,\textup{Berge-}C_3,\textup{Berge-}C_4\})=(1+o(1))\frac{1}{6}n^{3/2}$. By Proposition \ref{lvas} these two statements are equivalent.

\section{Acknowledgments}

The first author is supported in part by the J\'anos Bolyai Research Fellowship of the Hungarian Academy of
Sciences and the National Research, Development and Innovation Office -- NKFIH under the grants K
116769, KH 130371 and SNN 12936.}

\end{document}